\documentclass[11pt]{article}
\usepackage{mathrsfs}
\usepackage{amssymb}
\usepackage{amsmath}
\usepackage{amsthm}
\usepackage{amscd}
\usepackage{epstopdf}
\usepackage{stmaryrd}
\usepackage{stmaryrd}
\usepackage{enumerate}
\usepackage{hyperref}

\newtheorem{theorem}{Theorem}[section]
\newtheorem{conjecture}[theorem]{Conjecture}

\newtheorem{statement}[theorem]{Statement}

\newtheorem{Theorem}{Theorem}[section]
\newtheorem{Claim}[Theorem]{Claim}
\newtheorem{Lemma}[Theorem]{Lemma}
\newtheorem{Proposition}[Theorem]{Proposition}
\newtheorem{Corollary}[Theorem]{Corollary}
\newtheorem{definition}{Definition}

\DeclareMathOperator{\Spec}{Spec}
\DeclareMathOperator{\End}{End}
\DeclareMathOperator{\Der}{Der}

\DeclareMathOperator{\hol}{hol}
\DeclareMathOperator{\an}{an}
\DeclareMathOperator{\rig}{rig}
\DeclareMathOperator{\irr}{irr}
\newcommand{\Cbar}{\overline{\mathcal{C}}}
\newcommand{\Co}{\mathcal{C}^o}
\newcommand{\N}{\mathbb{N}}
\newcommand{\Q}{\mathbb{Q}}
\newcommand{\Z}{\mathbb{Z}}

\newcommand{\C}{\mathbb{C}}
\newcommand{\F}{\mathbb{F}}
\newcommand{\Xr}{\mathcal{X}_{r_1,r_2}}
\newcommand{\Yr}{\mathcal{X}_{r'_1,r'_2}}
\newcommand{\X}{\mathcal{X}}
\newcommand{\Or}{\mathcal{O}^{\rig}}
\newcommand{\bOr}{\mathcal{O}^{\rig}_{\le 1}}
\newcommand{\GL}{\rm{GL}}
\newcommand{\p}{\mathbb{P}}
\newcommand{\M}{\mathcal{M}}
\newcommand{\A}{\mathcal{A}}

\newcommand{\Mg}{\mathcal{M}_{g}}
\newcommand{\Mgd}{\mathcal{M}_{g,d}}
\newcommand{\Mgbar}{\overline{\mathcal{M}_{g}}}

\newcommand{\Mgdbar}{\overline{\mathcal{M}_{g,d}}}
\newcommand{\Modbar}{\overline{\mathcal{M}_{0,d}}}
\newcommand{\Mod}{\mathcal{M}_{0,d}}

\title{The $p$-curvature conjecture and monodromy about simple closed loops}
\author{Ananth N. Shankar}
\date{\today}

\begin{document}
\maketitle
\begin{abstract}
The Grothendieck-Katz $p$-curvature conjecture is an analogue of the Hasse Principle for differential equations. It states that a set of arithmetic differential equations on a variety has finite monodromy if its $p$-curvature vanishes modulo $p$, for almost all primes $p$. We prove that if the variety is a generic curve, then every simple closed loop on the curve has finite monodromy. 
\end{abstract}
\tableofcontents

\section{Introduction}
The Grothendieck-Katz $p$-curvature conjecture posits the existence of a full set of algebraic solutions to arithmetic differential equations in characteristic 0, given the existence of solutions after reducing modulo a prime for almost all primes. 

More precisely, let $R \subset \C$ be a finitely generated $\Z$-algebra. Suppose that $X$ is a smooth connected variety over $\Spec R$ and let $(V,\nabla)$ be a vector bundle on $X$ equipped with a flat connection. For almost every prime number $p$, we can reduce $X, V$ and $\nabla$ to obtain a vector bundle with connection on the $\Spec R/p$ scheme, $X/p$. Associated to such a system is an invariant $\psi_p$, its $p$-curvature, whose vanishing is equivalent to the existence of a full set of solutions to $\nabla$ modulo $p$. 

\begin{conjecture}[Grothendieck-Katz $p$-curvature conjecture]
Suppose that for almost all primes $p$, the $p$-curvature of $(V/p,\nabla/p)$ vanishes. Then the differential equation $(V,\nabla)$ has a full set of algebraic solutions, i.e. it becomes trivial on a finite etale cover of $X$. 
\end{conjecture}
Katz in \cite{Katz} proved that if the $p$-curvatures vanish for almost all $p$, then $(V,\nabla)$ has regular singular points, so that the conclusion of the conjecture is equivalent to asking that the monodromy representation of $(V,\nabla)$ is finite. Moreover, he showed that $(V,\nabla)$ has finite local mondromy. More precisely, if $X \hookrightarrow \overline{X}$ is a smooth compactification of $X$ with $\overline{X} \setminus X$ a normal crossings divisor, Katz proved that the monodromy about a boundary component is finite. 

The main theorem of this paper is:  
\begin{Theorem}\label{gen}
Let $(V,\nabla)$ be a vector bundle with connection on a generic proper smooth curve of genus $g$, such that almost all $p$-curvatures vanish. Then the monodromy about any simple closed loop is finite. 
\end{Theorem}
Here, by generic, we mean that the image of the map from $\Spec R$ to $\Mg$ (the moduli space of genus $g$ curves) contains the generic point of $\Mg$. We are also able to prove results for open curves:

\begin{Theorem}\label{P1}
Let $(V,\nabla)$ be a vector bundle with connection on $\mathbb{P}^1 \setminus \{Q_1, \hdots Q_d\}$ where $\{Q_1 \hdots Q_d\}$ is a generic set of $d$ points, $d >2$. If the $p$-curvatures vanish for almost all primes $p$, then the monodromy about any simple closed loop is finite.  
\end{Theorem}
We note that Theorem \ref{P1}, when applicable, generalises Katz's local monodromy theorem. Katz proves that any simple closed loop which encloses exactly one point has finite order. However, nothing is said about loops which enclose more than just a single point. Theorem \ref{P1} shows that any simple closed loop, though it may bound several points, has finite order in the monodromy representation, so long as the points are generic. 

Another result, which combines Theorems \ref{gen} and \ref{P1} is

\begin{Theorem}\label{genpts}
Let $(V, \nabla)$ be a vector bundle with connection on the generic $d$-punctured curve. If the $p$-curvatures vanish for almost all primes $p$, then the monodromy about every simple closed loop is finite. 
\end{Theorem}

We prove Theorems \ref{gen}, \ref{P1} and \ref{genpts} by using a beautiful suggestion of Emerton and Kisin, which allows us to view global monodromy as local monodromy. Given a family of smooth algebraic curves degenerating to a nodal curve, the singular curve can be realised by contracting an appropriately chosen simple closed loop in one of the smooth fibers. Indeed, locally about the node, we get a family of holomorphic annuli degenerating to the node. The annuli approximate a punctured disc, and this approximation gets better and better as one approaches the node. 

With this idea as our guiding principle, we prove that when given a vector bundle on the smooth part of such a family with a connection relative to the base, the $p$-curvatures vanishing imply the finiteness of monodromy about the vanishing loop. The precise result is stated as Theorem \ref{thm:tech}. The key step is to prove a rigid-analytic version of the $p$-curvature conjecture: %We prove that if the $\textrm{mod}$--$p$ reductions of $(V,\nabla)$ over an equicharacteristic-zero annulus have full sets of solutions, then $(V, \nabla)$ has finite monodromy. Indeed, suppose that $\mathcal{C} \rightarrow B$ is the family of curves, and that $Y \subset B$ is an irreducible component of the nodal locus. We may assume that $Y$ is generically smooth, and that it is of codimension one. Let $F$ and $K$ be the function fields of $B$ and $Y$ respectively. Let $q \in F$ be a function whose vanishing locally defines $Y$. This defines a discrete valuation on $F$, with respect to which we complete $F$, to obtain a complete discrete valuation field isomorphic to $K((q))$. The formal completion of the curve at its node is a rigid-analytic annulus $\X^o$ with coordinate $t$, over $K((q))$. The annulus $\X^o$ is the algebraic analogue of the family of holomorphic annuli mentioned above. We prove that $(V,\nabla)$ pulled back to an appropriate closed subannulus $\X \subset \X^o$ has finite monodromy. 
suppose $S$ is a finitely-generated $\Z$-algebra, and let $(V,\nabla)$ be a vector bundle with connection on a closed rigid-analytic annulus $\X$ (with coordinate $t$) over $S((q))$. We also assume that $(V,\nabla)$ has a cyclic vector. If the $p$-curvatures vanish for almost all primes $p$, we prove that $(V,\nabla)$ pulled back to $\X \times_{S((q))} K((q))$ has finite finite monodromy, where $K$ is the fraction field of $S$. 
The exact result we prove is stated as Theorem \ref{thm:use}. 

This allows us to prove that $(V,\nabla)$ pulled back to the family of holomorphic annuli has finite monodromy, thereby completing the proof of Theorem \ref{thm:tech}. We remark that this is an instance where the theory of rigid-analytic varieties over an equicharacteristic-zero base is used to prove a result in arithmetic. 

We note that Katz's theorem on local monodromy cannot be applied to prove Theorem \ref{thm:use}. Indeed, in Katz's setting, all the singularities are poles of finite order (in $t$), whereas we have to deal with series in $t$ with infinitely many negative powers. A key step is to prove that $(V,\nabla)$ extends to an integral model of $\X$ which has special fiber $\mathbb{G}_m$. 

The paper is organised as follows. In \S 2, we state Theorem \ref{thm:tech} precisely and infer from it Theorem \ref{gen}. We spend Sections 3--5 proving Theorem \ref{thm:tech}. In \S3, we formulate and prove Theorem \ref{thm:use}, and a more general result. In \S4, we give a rigid criterion for a family of connections on holomorphic annuli to have solutions. In \S5, we put together the results in the previous two sections to finish the proof of Theorem \ref{thm:tech}. Finally in \S6, we demonstrate the proof of the other applications of Theorem \ref{thm:tech}.

\subsection*{Acknowledgements}
It is a pleasure to thank my advisor Mark Kisin, for suggesting that I work on this problem, and also for many conversations and ideas. I also thank George Boxer, Anand Patel and Yunqing Tang for several helpful discussions. Finally, I thank Shiva Shankar and Yunqing Tang for reading previous versions of this paper, and for offering very helpful comments. 

\section{Statement of the main result}
\subsection{$p$-curvature}
Let $R_p$ be an $\F_p$-algebra, and let $X$ be a scheme smooth over $\Spec R_p$. Suppose that $(V,\nabla)$ is a vector bundle on $X$ with flat connection relative to $R_p$, and let $\mathcal{N}$ be its sheaf of sections. Suppose that $D$ is a section of $\Der(X/R_p)$, the sheaf of $R_p$-linear differentials on $X_p$. The data of $\nabla$ gives, for every $D$, a map 
$$\nabla(D): \mathcal{N} \rightarrow \mathcal{N}.$$
This map is $R_p$-linear, and obeys the Liebnitz rule. The $p$-curvature $\psi_p$ is defined to be the mapping of sheaves
$$\Der(X/R_p) \rightarrow \End_{\mathcal{O}_X}(\mathcal{N})$$
by setting 
$$\psi_p(D) = \nabla(D)^p - \nabla(D^p).$$
Note that $D^p$ is a derivation, as we are in characteristic $p >0$. It is easy to see that $\psi_p(D)$, which apriori lies only in $\End_{R_p}(\mathcal{N})$, is actually an element of $\End_{\mathcal{O}_X}(\mathcal{N})$. It is a well known fact due to Cartier (section 5.2 in \cite{Katz}) that $\psi_p$ being identically zero is equivalent to $\nabla$ having a full set of solutions.

\subsection{Simple closed loops}
Let $C$ be a smooth algebraic curve (not necessarily projective) over $\Spec \C$, and let $C^{\an}$ denote the associated complex analytic space. 
\begin{definition}
We say that $\gamma \subset C^{\an}$ is a simple closed loop if it is the image of an injective continuous map $f_{\gamma}: S^1 \rightarrow C^{\an}$. 
\end{definition}
Given a vector bundle $V$ with connection $\nabla$ on $C^{\an}$, we say that $\gamma$ has finite order in the monodromy representation if for some $b \in \gamma$, $f_{\gamma} \in \pi_1(C^{\an},b)$ has finite order in the monodromy representation based at $b$. Note that this definition is independent of $b \in \gamma$, and the parameterization map $f_{\gamma}$. 

\begin{definition}
Two simple closed loops $\gamma$ and $\gamma'$ of $C^{\an}$ are isotopic if there is a map $I:S^1 \times [0,1]\rightarrow C^{\an}$ with the following properties: 
\begin{enumerate}
\item $I(S^1 \times \{0\}) = \gamma$ and $I(S^1 \times \{1\}) = \gamma'$
\item $I (S^1 \times \{a\}) \rightarrow C^{\an}$ is an embedding for $a \in [0,1]$. 
\end{enumerate}
\end{definition} 

We now make precise what we mean for a simple closed loop to contract to a node in a family of curves. To that end, suppose that $\overline{B}$ is a variety over $\Spec \C$. Let $\mathcal{C} \rightarrow \overline{B}$ be a family of algebraic curves over $\overline{B}$ which is generically smooth.  Let $B \subset \overline{B}$ be the open subvariety over which $\mathcal{C}$ is smooth. Suppose that there exists a closed point $s_0 \in \overline{B} \setminus B$ such that the fiber $\mathcal{C}_{s_0}$ over $s_0$ is a nodal curve, with $P \in \mathcal{C}_{s_0}$ a node. Let $s \in B(\C)$ and suppose that $\gamma \subset \mathcal{C}^{\an}_s$ is a simple closed loop. 

\begin{definition}
We say that $\gamma$ contracts to the node $P$ if there exists a continuous map $J: S^1 \times [0,1] \rightarrow \mathcal{C}^{\an}$ such that 
\begin{enumerate}
\item For $a<1$, $J(S^1 \times a)$ is a simple closed loop contained entirely within a single smooth fiber of $\mathcal{C}^{\an}$. 
\item $J(S^1 \times \{0\}) = \gamma$ and $J(S^1 \times \{1\}) = \{P\}$.
\end{enumerate}
\end{definition}
Clearly $J$ induces a map $\beta: [0,1] \rightarrow \overline{B}^{\an}$, such that $[0,1)$ is contained in the interior $B^{\an}$. We say that $\gamma$ contracts to the node via $\beta$. It follows from the definitions that if $\gamma$ and $\gamma'$ are isotopic simple closed loops contained in some smooth fiber, $\gamma$ contracts to the node if and only if $\gamma'$ does. 

\subsection{The generic curve}
Suppose that $\overline{B}$ is a variety over a number field. Let $\mathcal{C} \rightarrow \overline{B}$ be as above. We now state the main technical result, from which Theorem \ref{gen} follows. 

\begin{Theorem}\label{thm:tech}
Notation as above. Suppose that $(V, \nabla)$ is a vector bundle on $\mathcal{C}_{|_B}$ with connection relative to $B$, such that almost all $p$-curvatures vanish. Let $s \in B(\C)$. Suppose that $\gamma\subset \mathcal{C}_s^{\an}$ is a simple closed loop which contracts to the node. Then the image of $\gamma$ in the monodromy representation has finite order. 
\end{Theorem}

We end this section by demonstrating how Theorem \ref{thm:tech} is used to prove Theorem \ref{gen}. We first prove the following lemma, which will allow us to reduce Theorem \ref{gen} to the case of {\it the} generic curve. Let $X$ and $Y$ be irreducible varieties over $\C$, and let $f: X \rightarrow Y$ be a map with generically connected fibers. Suppose that $U \subset X$ and $V \subset Y$ denote smooth open subvarieties of $X$ and $Y$, such that $f$ restricted to $U$ is a smooth morphism. Suppose that $x \in U^{\an}$, and $y = f(x)$. 

\begin{Lemma}\label{pathlifting}
Let $X, Y,f$ be as above, and let $\beta: [0,1] \rightarrow X^{\an}$ be a path, such that $\beta(0) = x$ and $\beta([0,1)) \subset U^{\an}$. Let $\alpha: [0,1] \rightarrow V^{\an}$ be any loop based at $y$. Then, a suitable reparameterization of the path $\alpha * (f \circ \beta)$ can be lifted to $X$.
\end{Lemma}
\begin{proof}
The path $\alpha$ can be lifted to a path $\tilde{\alpha}: [0,1] \rightarrow X$ because the map $f: U \rightarrow V$ is smooth. Because the fibers of $f$ are connected, there exists a path $\alpha'$ contained inside $f^{-1}(y)^{\an}$ which connects $\tilde{\alpha}(0)$ and $x$. The path $\tilde{\alpha}*\alpha'*\beta$ is clearly a lift of $\alpha * (f \circ \beta)$ (up to reparameterization).
\end{proof}

\begin{proof}[Proof of Theorem \ref{gen}]
If $g = 0$, then the curve is simply connnected, and so there is no monodromy. If $g = 1$, then fundamental group is abelian, and by \cite{Bost} the monodromy represenation has finite image. Therefore, we assume that $g \geq 2$. We deal with the cases $g =2$ and $g > 2$ separately, because the generic point of the moduli space of genus two curves is not fine.
\subsubsection*{Case 1: $g > 2$}

Let $\Mg \rightarrow \Q$ be the moduli space of genus $g$ curves. Let $\Mgbar$ be its compactification (as constructed by Knudsen in \cite{Knudsen}. Also see \cite[Chapter 4]{Harris}). We recall the following facts about $\Mg, \ \Mgbar$ and mapping class groups of curves. 
\begin{enumerate}
\item The generic genus $g$ curve has no automorphisms, for $g>2$. \label{itm:one}
\item There exists a boundary divisor $D_{\irr} \subset \Mgbar$, consisting of irreducible nodal curves. The generic irreducible nodal curve has no automorphisms.  \label{itm:two}
\item For every $g_1 + g_2 = g$ with $g_i \geq 1$, there exists a boundary divisor $D_{g_1,g_2}$ consisting of a reducible curve, with two smooth irreducible components with genera $g_1$ and $g_2$ meeting at a node. If the $g_i >1$, the generic such curve has no automorphisms. \label{itm:three}
\item By Facts \ref{itm:one}, \ref{itm:two} and \ref{itm:three}, there exists an open subvariety $\overline{U} \subset \Mgbar$ which contains the generic point of $D_{\irr}$ and $D_{g_1,g_2}$ for $g_i >1$ over which the moduli problem is fine. Let $\mathcal{C} \rightarrow \overline{U}$ be the corresponding family of curves. Let $\overline{U} \cap \Mg = U$ ($U$ can be assumed to be affine). \label{itm:four}
\item For $s \in U(\C)$ let $\Sigma_s$ denote the mapping class group of $\mathcal{C}_s^{\an}$. There is a surjective map from $\pi_1(U(\C),s)$ to $\Sigma_s$. The action of $\pi_1(U(\C),s)$ on the isotopy classes of simple closed loops (by parallel transport) of $\mathcal{C}_s$ factors through this map. \label{itm:five}
\item The mapping class group $\Sigma_s$ acts transitively on the isotopy classes of non-separating simple closed loops. If two separating simple closed loops break $\mathcal{C}_s$ into pieces with same genera, then the isotopy classes of these two loops lie in the same $\Sigma_s$ orbit. (See \cite{Farb})\label{itm:six}
\end{enumerate}
We will first prove the theorem for {\it the} generic curve, i.e. we will prove that every simple closed loop of $\mathcal{C}_s$ has finite order in the monodromy representation. There exists a proper normal irreducible variety $\overline{\M'}$ which is a double-cover of $\Mgbar$ branched over the generic point of $D_{1.g-1}$, and which is \`etale over the generic points of all the other boundary divisors. Let $D' \subset \overline{\M'}$  denote the pullback of $D_{1,g-1}$. The family of curves over $\overline{\M'}$ extends to the generic point of $D'$. In order to prove the result for the generic curve, it suffices to prove it over the family over $\overline{\M'}$. Let $U' \subset \overline{\M'}$ denote the pullback of $U \subset \Mgbar$. 
We now apply Theorem \ref{thm:tech} as follows. 

Suppose that $s' \in U'(\C)$. Let $s \in U(\C)$ be its image. Let $\beta$ be a path joining $s \in U(\C)$ and $s_0 \in D_{\irr}(\C) \cap \overline{U}(\C)$. This corresponds to a non-separating simple closed loop $\gamma \subset \mathcal{C}_s$ contracting to the nodal point of $\mathcal{C}_{s_0}^{\an}$ along $\beta$. By applying Theorem \ref{thm:tech}, we see that $\gamma$ acts with finite order. By Facts \ref{itm:five} and \ref{itm:six}, the action of $\pi_1(U(\C),s)$ on the set of isotopy classes of non-separating simple closed loops is transitive. Therefore, by modifying $\beta$ by a loop $\alpha$ based at $s$, every non-separating simple closed loop $\gamma' \subset \mathcal{C}_s^{\an}$ can be contracted to the node of $\mathcal{C}_{s_0}^{\an}$. Lifting this path to $\overline{\M'}$, we see that every simple closed loop in $\mathcal{C}_{s'}$ can also be contracted to the node of an boundary point
of $\overline{\M'}$. The result follows. 

For a simple closed loop which breaks $\mathcal{C}_{s'}^{\an}$ into two pieces both of which have genus greater than $1$, the same argument applies. 

Finally, we claim that the fundamental group of $U'(\C)$ acts transtively on the isotopy-classes of simple closed loops on $\mathcal{C}_{s'}^{\an}$, which break $\mathcal{C}_{s'}^{\an}$ into two pieces, one of which has genus $1$. Indeed, the element of $\pi_1(U(\C),s)$ which corresponds to local monodromy about $D_{1,g-1}$ acts trivially on such loops, and along with the identity, forms a full set of representatives for the quotient $\pi_1(U'(\C),s')$, establishing the claim. Therefore, the above argument applies to prove the result in the case of the family over $\overline{\M'}$, and therefore for the generic family. 

Now, let $\mathcal{C} \rightarrow \Spec R$ be any any generic proper smooth genus $g$ curve. We replace $\Spec R$ with $\Spec R \times_{\Mgbar} U'$ (note that $U'$ is affine, and therefore, $\Spec R \times_{\Mgbar} U'$ is affine. For ease of notation, we refer to this as $\Spec R$). There exists a proper normal irreducible variety $X$ along with a map $f: X \rightarrow \overline{\M'}$ with the following properties: 
\begin{enumerate}
\item The map $\Spec R \rightarrow \overline{\M'}$ factors through $X \rightarrow \overline{\M'}$. 
\item The map $\Spec R \rightarrow X$ is generically finite.
\item The map $X \rightarrow \overline{\M'}$ has generically connected fibers. 
\end{enumerate}
By replacing $\Spec R$ with a birational variety, we may assume that $X$ is a proper normal variety. Let $x \in X(\C)$ denote a point at which $f$ is smooth (such an $x$ exists because $f$ is generically smooth), and let $C$ denote the fiber of $\mathcal{C}$ over $x$. Given any simple closed curve, we may apply Lemma \ref{pathlifting} to find a path $\beta \subset X^{\an}$ which contracts the chosen simple closed curve to a point. Finally, we replace $\Spec R$ with a normal proper variety $X'$ birational to it. Therefore, the map $X' \rightarrow X$ is generically \`etale, and is branched over the generic points of the boundary divisors of $X$. By using the path lifting property for branched covers, we see that every simple closed curve as above on a generic fiber over $X'$ can be contracted to a point. We apply Theorem \ref{thm:tech} to conclude the finiteness of monodromy. 

\subsubsection*{Case 2: $g = 2$}
Consider the family of projective curves $\mathcal{C} \rightarrow \mathcal{M}_{0,6}$ with affine equations 
$$y^2 = x(x-1)(x-\lambda_1)(x-\lambda_2)(x-\lambda_3).$$
The same arguments used in Case 1 for the family $\overline{\M'}$ apply to prove Theorem \ref{gen} for this family. Suppose that $\mathcal{C} \rightarrow H$ was some other family of smooth genus 2 curves over an irreducible base $H$. We claim that there exists a generically \`etale map $H' \rightarrow H$, such that the family $\mathcal{C} \times_H H'$ is pulled back from $\mathcal{M}_{0,6}$. The arguments of Case 1 again apply to prove our result if this were the case. 

We now prove the claim. Because $\mathcal{C} \rightarrow H$ is a family of genus two curves, it is hyperelliptic Zariski-locally over $H$, i.e. there exists a degree-2 map $\mathcal{C} \rightarrow H \times \mathbb{P}^1$, once we replace $H$ by a Zariski-open subset. The inverse image of the ramification locus in $H \times \mathbb{P}^1$ gives a degree-six \`etale cover of $H$. By choosing $H'' \rightarrow H$ to be an appropriate \`etale cover and pulling the family back to $H''$, we may assume that the degree-six \`etale cover has six connected components. This yields a canonical map $H'' \rightarrow \mathcal{M}_{0,6}$, and the family $\mathcal{C} \rightarrow H''$ is locally a quadratic twist of the pullback family from $\mathcal{M}_{0,6}$. By shrinking $H''$ if necessary, and taking a double-cover $H'$ of $H''$, we can get rid of the quadratic twist to assume that $\mathcal{C} \times_{H''} H'$ is indeed the pullback family from $\mathcal{M}_{0,6}$, as required. The theorem follows. 

\end{proof}

\section{Rigid computations}
In this section, we work with rigid-analytic varieties over equicharacteristic bases. All our spaces are of dimension 1, and a good reference for this case is \cite{Kedlaya}. For a treatment of higher dimensinal varieties, see \cite{Remmert} and \cite{Tate}.

Let $S \subset \C$ be a $\Z$-algebra of finite type and let $K$ be its field of fractions. By inverting finitely many elements of $S$, we may assume that $S \cap \overline{\Q} = \mathcal{O}$. Here, $\mathcal{O}$ is the ring of integers of some number field, sufficiently localized. Then for almost all primes $\mathfrak{p} \subset \mathcal{O}$, $S/\mathfrak{p}$ is also an integral domain, with fraction field $K_{\mathfrak{p}}$ of same transcendence degree over $\mathbb{F}_p$ as $K$ over $\mathbb{Q}$. Given an element of $K$, its $\mathfrak{p}$-adic valuation still makes sense, and if this valuation isn't negative, it makes sense to reduce the element modulo $\mathfrak{p}$ to obtain an element of $K_{\mathfrak{p}}$. 

Consider the ring of power series over $K$ (\textit{resp.} $K_{\mathfrak{p}}$) in the variable $q$, $K[|q|]$ (\textit{resp.} $K_{\mathfrak{p}}[|q|]$), and its field of fractions $K((q))$ (\textit{resp.} $K_{\mathfrak{p}}((q))$). These fields are complete with respect to the discrete $q$-adic valuation (which is non-archemedian). We consider vector bundles with connection on rigid-analytic annuli over $K((q))$ and $K_{\mathfrak{p}}$. We make the following definitions: 

\begin{definition}
\begin{enumerate}
\item $\Xr\ ($\textit{resp.} $\mathcal{X}_{r_1,r_2,\mathfrak{p}})$ is defined to be the closed rigid annulus over $K((q))\ ($\textit{resp.} $K_{\mathfrak{p}}((q)))$ with inner radius $r_1$ and outer radius $r_2$. 
\item $\X\ ($\textit{resp.} $\mathcal{X}_{\mathfrak{p}})$ is defined to be the closed rigid annulus over $K((q))\ ($\textit{resp.} $K_{\mathfrak{p}}((q)))$ with inner and outer radii 1.
\item For any rigid-analytic annulus $\mathcal{Y}$ as above, $\Or(\mathcal{Y})$ is defined to be the corresponding affnoid ring (of rigid functions).
\item For $\mathcal{Y}$ as above, $\bOr(\mathcal{Y})$ is defined to be that subring of $\Or(\mathcal{Y})$ consisting of functions bounded by $1$

\end{enumerate}
\end{definition}

For example, $\Or(\X) = K((q))\langle t^{-1},t \rangle$ is the ring of infinitely tailed laurent series with coefficients in $K((q))$ in the variable $t$ with the property that the coefficients at either end converge to zero. The ring of functions bounded by $1$ is $\bOr(X) = K[|q|] \langle t^{-1},t\rangle$, the ring of inifinitely tailed laurent series with coefficients in $K[|q|]$ in the variable $t$ with the property that the coefficients at either end converge to zero. Similarly, $\Or(\Xr) \subset K((q))[|t^{-1},t|]$ consists of elements of the form $\sum a_kt^k$ with the property that for all $k$, $|a_k|r_i^k \rightarrow 0$ as $|k| \rightarrow \infty$ for $i = 1,2$. A function $f$ is in $\bOr(\Xr)$ if it is in $\Or(\Xr)$ and satisfes the following (for all $k$, and for $i=1,2$): 
 \begin{equation}\label{growthcond}
|a_{k}|r_i^k \le 1
\end{equation}
More often than not, we will use the language of the above inequalities, instead of saying that $f \in \bOr$. 

In this section, the vector bundles with connection we consider will be those ``which can be reduced modulo $\mathfrak{p}$" (we will soon make precise what we mean by this) for almost all primes $\mathfrak{p} \subset \mathcal{O}$, and prove a rigid-analytic version of the Grothendieck-Katz $p$-curvature conjecture.

Before stating the main theorem of this section, we remark that we will exclusively work with annuli with ``rational" radii, i.e. with inner and outer radius a rational power of $|q|$. By extending scalars to fields obtained by adjoining rational powers of $q$, and by substituting $q^{\alpha}t$ for $t$ (for rational $\alpha$), we will be able to move back and forth between different annuli, without disrupting any of the integrality (with respect to primes $\mathfrak{p}$) properties we might have assumed. Henceforth, by annulus, we will always mean annulus with ``rational" inner and outer radii. Further, whenever we move back and forth between different annuli, we will leave implicit the operations of adjoining appropriate roots of $q$, and the replacing of $t$ with the appropriate $q^{\alpha}t$. 

We now state the main theorems of this section: 

\begin{Theorem}\label{thm:main}
Let $(V,\nabla)$ be a vector bundle with connection on $\Xr$, and suppose that there exists a cyclic vector with respect to which all data can be reduced modulo $\mathfrak{p}$ for almost all primes $\mathfrak{p} \subset \mathcal{O}$. If the $p$-curvatures vanish for almost all $p$, then there exists a basis of flat sections after finite pullback.
\end{Theorem}

\begin{Theorem}\label{thm:use}
Let $(V,\nabla)$ be a vector bundle with connection on $\X$, and suppose that there exists a cyclic vector with respect to which all data can be reduced modulo $\mathfrak{p}$ for almost all primes $\mathfrak{p} \subset \mathcal{O}$. Suppose that the $p$-curvatures vanish for almost all $p$. Then: 
\begin{enumerate}
\item The coefficient matrix of the connection with respect to the cyclic basis has entries in $\bOr(\X)$.
\item After finite pullback, there exists a basis of solutions over $\bOr(\X)$. 
\end{enumerate}
\end{Theorem}
Notice that $\bOr(\X)/(q) = K[t,t^{-1}]$, which is the coordinate ring of $\mathbb{G}_m$ over $K$. Part 1 of Theorem \ref{thm:use} proves that $(V, \nabla)$ actually extends to a very special integral model, namely the one with special fiber $\mathbb{G}_m$. Part 2 posits the existence of solutions over this integral model. 

We now say a word about what we mean by reducing reducing data modulo primes. Given a vector bundle with connection over a rigid $K((q))$-annulus, we would like to produce vector bundles with connections over the corresponding rigid $K_{\mathfrak{p}}((q))$-annuli, for almost all $\mathfrak{p}$. First of all, note that vector bundles over closed rigid-analytic annuli are always trivial, because the ring of functions is a principal ideal domain (for instance, see \cite[Proposition 7.3.2]{Kedlaya}). Therefore, given a basis, it makes sense to reduce the vector bundle modulo $\mathfrak{p}$ by simply looking at trivial vector bundles on the appropriate annulus over $K_{\mathfrak{p}}((q))$, with a basis of sections indexed by the same set. However, if we now add a connection, matters are not as clear. Fixing a basis and writing out the connection in terms of that basis (and the ubiquitous coordinate $t$), it might not be possible to reduce the connection modulo these primes.

 We say that $a \in K((q))$ can be reduced modulo almost all primes if its $\mathfrak{p}$-adic valuation is non-negative for almost all primes. We say that $f \in \Or(\Xr)$ can be reduced modulo almost all primes, if the set of primes $\mathfrak{p}$ with the property that some coefficient $a$ of $f$ has negative $\mathfrak{p}$-adic valuation, is finite. Here, we think of $f$ as an infinitely tailed Laurent series in the coordinate $t$ with coefficients in $K((q))$. In this case, $f$ can be reduced modulo almost all primes $\mathfrak{p}$ to give functions on $\X_{r_1,r_2,\mathfrak{p}}$. Also, note that a function $f$ which is bounded by $1$ will reduce to a bounded-by-$1$ function on the  corresponding mod-$\mathfrak{p}$ annulus. In fact, if the reduction of $f$ modulo almost all primes is bounded by $1$, then $f$ will also be bounded by $1$. 

Henceforth, we will always work with the data of a vector bundle with connection along with a basis of sections of the vector bundle, such that the connection matrix with respect to that basis can be reduced modulo almost all primes. Once we have a vector bundle with connection on a $K_{\mathfrak{p}}((q))$-annulus, the notion of its $p$-curvature makes complete sense. When we talk of a vector bundle with connection having vanishing $p$-curvatures, we implicitly mean that the basis which we're working with has the property that the connection matrix can be reduced modulo almost all primes $\mathfrak{p}$, and that the induced vector bundle with connection over $K_{\mathfrak{p}}$-annuli have vanishing $p$-curvatures. 

\subsection{Extension to an integral model}
For the rest of this section, denote by $D$ the derivation $t\frac{d}{dt}$. In this subsection, we prove a statement slightly more general than the first statement in Theorem \ref{thm:use}. 
\begin{Proposition}
Let $(W, \nabla)$ be a vector bundle with connection along with a cyclic vector $v$ on $\X$, such that the $p$-curvatures vanish for infinitely many $p$. Then, the entries of the connection matrix (with respect to the above cyclic basis) lie in $\bOr(\X)$. 
\end{Proposition}
Note that we only need infinitely many $p$-curvatures to vanish. 
\begin{proof}
Let the $n$ be the dimension of $V$. Suppose that the connection matrix (with respect to our cyclic basis $\bf{e}$) is 
\begin{equation*}
A=
\left(
\begin{array}{cccc}
&&& f_0\\
1& &&f_1\\
&\ddots & &\vdots\\
&&1&f_{n-1}
\end{array}
\right)
\end{equation*}

Let $f_k(t) = \sum a_{k,i}t^i$. By replacing $t$ by $t^{1/n!}$, we may assume that $a_{k,i} = 0$ unless $n!\vert i$. We also assume that some $a_{k,i}$ has size bigger than 1. Let $\ell = \max_{k,i}{ |a_{k,i}|^{\frac{1}{n-k}}}$, and let $\nu = \max_{|a_{k,i}| = \ell^{n-k}}{ (\frac{i}{n-k})}$. 
Consider the following basis ${\bf f}$ of the vector bundle
\begin{equation*}
{\bf f} = 
\left(
\begin{array}{cccc}
1&&&\\
&t^{-\nu}&&\\
&&\ddots&\\
&&&t^{-(n-1)\nu}
\end{array} 
\right)
{\bf e} = T {\bf e}
\end{equation*}

A direct computation shows that 
\begin{equation*}
\nabla(D){\bf f} = B{\bf f}
\end{equation*}
where 
\begin{equation*}
B=
\left(
\begin{array}{cccc}
0&&&\\
&-\nu&&\\
&&\ddots&\\
&&&-(n-1)\nu
\end{array}
\right)
+
\left(
\begin{array}{cccc}
&&& t^{-(n-1)\nu}f_0\\
t^{\nu}&&&t^{-(n-2)\nu}f_1\\
&\ddots & &\vdots\\
&&t^{\nu}&f_{n-1}
\end{array}
\right)
\end{equation*}

The entries of $B$ can still be reduced modulo almost all primes $\mathfrak{p} \subset \mathcal{O}$. Let $t^{-(n-k -1)\nu}f_k(t) = \sum b_{k,i}t^i$. The definition of $\nu$ implies that $|b_{k,i}| < \ell^{n-k}$ for $i > \nu$, and that $|b_{k,i}| \le \ell^{n-k}$, with equality holding for at least some pair $(k, -\nu)$. Write the connection matrix $B$ as $\sum\nolimits_{i \in \mathbb{Z}}t^i B_i$. The $B_i$ have the property that as $|i| \rightarrow \infty$, $B_i \rightarrow \bf{0}$. For $i$ different from $0, \nu$, $B_i$ has non-zero entires only in its last column. We have that \\
\begin{equation*}
B_0=
\left(
\begin{array}{cccc}
0&&&b_{0,0}\\
&-\nu&&b_{1,0}\\
&&\ddots&\\
&&&-(n-1)\nu + b_{n-1,0}
\end{array}
\right)
\end{equation*}
\\and \\
\begin{equation*}
B_{\nu}=
\left(
\begin{array}{cccc}
&&& b_{0,\nu}\\
1& &&b_{1,\nu}\\
&\ddots & &\vdots\\
&&1&b_{n-1,\nu}
\end{array}
\right)
\end{equation*}
\\

We now reduce modulo a prime $\mathfrak{p} \subset \mathcal{O}$, large enough so that the sizes of the $b_{k,\nu}$ don't change when reduced modulo $\mathfrak{p}$. Suppose also, that the $p$-curvature vanishes, and that $p$ is larger than $2n$. For ease of notation, we use the same terms to denote the connection matrix and its entries modulo $\mathfrak{p}$. 

The characteristic polyomial of $B_{\nu}$ splits over a separable extension of $K_\mathfrak{p}((q))$ (because $p>2n$). Therefore, by \cite[Lemma 3]{Kedlaya1}, there exists an integer $m$ such that the field $L = \overline{K_\mathfrak{p}}((q^{1/m}))$ contains all the eigenvalues of $B_{\nu}$. 

In the remainder of the proof, we establish the following statement, which supplies the necessary contradiction. 

\begin{statement}
There exists a vector $w$ with entries in L, such that $\nabla(D^p)(w) \neq \nabla(D)^p(w)$. 
\end{statement}

 Let $\lambda$ be an eigenvalue with the largest norm. Clearly, $|\lambda| = \ell$. Let $w = [\theta_0, ... , \theta_{n-1}]^t$ be an eigenvector with eigenvalue $\lambda$. We may assume that $|\theta_{n-1}|=1$, beacuse the last coordinate of an eigenvector of $B_{\nu}$ with non-zero eigenvalue cannot be $0$. It is easy to check that $|\theta_k| \le \ell^{n-1-k}$. 

Let $\nabla(D)^p w = \sum w_i t^i$. We describe $w_{\nu p}$ as a sum of terms as follows: 
\begin{equation*}
w_{\nu p}t^{\nu p} = \sum_{W \in \mathcal{I}} Ww
\end{equation*}
where $\mathcal{I}$ is the set of all length $p$ words in the letters $B_i$ and $D$, with the property that the sum of all the subscripts $i$ occuring in $W$ is $\nu p$. The word $W$ acts on $w$ as follows: $B_i$ multiplies vectors by $B_i t^i$, and $D$ just acts on the power of $t$ in the usual way.

For a vector $v$, let $v[k]$ denote its $(k+1)^{th}$ entry (therefore, the last entry of $v$ would be denoted by $v[n-1]$). We will show that the word $W_0 = B_{\nu}B_{\nu} \hdots B_{\nu}$ has the property that $W_0w[n-1]$ is strictly larger than $Ww[n-1]$, $W_0 \neq W\in \mathcal{I}$. To that end, we make the following claim, which we prove after finishing the proof of this propostion.

\begin{Claim}
Suppose that $vt^a = [\phi_0, ... ,\phi_{n-1}]^t$ is a vector with the property that for all $k$, $|\phi_k| \le$ $(resp.$ $<)$ $\ell^{n-k-1} |\theta_{n-1}|$. Then, the coordinates of $\lambda w$  and $C v$ satisfy the same inequalities, where $C$ stands for either $B_i$, or $D$. 
\end{Claim}

The proposition follows from the claim as follows: let  $w^j$ and $w_0^j$ be the vectors obtained by applying the first $j$ letters of $W$ and $W_0$ respectively on $w$. The claim implies $|w^j[k]| \le \ell^{n-k-1}|w_0^j[k]|$ for any $0\le k \le n-1$. However, $W$ differing from $W_0$, must contain the letter $D$, or a letter of the form $B_i$ where $i>\nu$. In either case, such a letter would render the inequality strict, i.e. if such a letter first occured at the $j_0^{th}$ stage, then $|w^{j_0}[k]| < \ell^{n-k-1}|w_0^{j_0}[k]|$. According to the claim, this said strictness would persist through the application of the rest of the word, i.e. $|w^{j}[k]| < \ell^{n-k-1}|w_0^{j}[k]|$, for $j \ge j_0$. Therefore, $|w_{\nu p}[n-1]| = |B_{nu}^p w [n-1]| = \ell^p$. All the entries of $\nabla(D)w$ are bounded by $\ell^{2n}$, which establishes the statement, as required.
\end{proof}

\begin{proof}[Proof of Claim]
The claim is clearly true if $C$ stands for $D$, because $\ell > 1$ and applying $D$ doesn't increase the size of any of the coordinates. For the rest of the proof, we ignore the factor of $t$, as it changes nothing. If $C$ stands for $B_i$ for $i$ different from $0$ and $-\nu$, then $|Cv[k]| = |b_{i,k}\phi_{n-1}| \le$ $(resp. $ $<) \ell^{n-k}|\theta_{n-1}|$, as required. If $L = B_0$, $Cv[k] = -k\nu\phi_k + b_{k,0}\phi_{n-1}$, which is less than $\ell^{n-k}|\theta_{n-1}|$ because $|\phi_{n-1}| \le$ $(resp.$ $<) |\theta_{n-1}|$ and $|b_{k,0}| \le \ell^{n-k}$, as required.

Finally, suppose that $C = B_{\nu}$. For brevity, let $b_{k}=b_{k,-\nu}$. We have that 
\begin{equation*}
Cv = [b_0\phi_{n-1}, \phi_0 + b_1\phi_{n-1}, ... ,\phi_{n-2} + b_{n-1}\phi_{n-1}]^t.
\end{equation*}
We must show that $|Cv[k]| \le$ $(resp.$ $<)$ $ \ell^{n-k}|\theta_{n-1}|$. This is clear for $Cv[0]$. We have $Cv[k] = \phi_{k-1} + b_k\phi_{n-1}$. The bound on $b_k$ shows that $|b_k\phi_{n-1}| \le$ $(<)$ $\ell^{n-k}|\theta_{n-1}|$. Therefore it suffices to show that $|\phi_{k-1}| \le$ $(<)$ $\ell^{n-k}|\theta_{n-1}|$. This proves the claim, because  $\phi_{k-1}$ satisfies the required inequality by hypothesis.
\end{proof}

By appropriately substituting $q^{\alpha}t$ for $t$, the corollary below directly follows:
\begin{Corollary}\label{grow}
Let $(V, \nabla)$ be a vector bundle with connection along with a cyclic vector $v$ on $\Xr$, such that the $p$-curvatures vanish for infinitely many $p$. Then, the entries of the connection matrix (with respect to the above cyclic basis) satisfy $\eqref{growthcond}$ and therefore lie in $\bOr(\Xr)$. 
\end{Corollary}

\subsection{Existence of rigid solutions in a special case}
We devote this subsection to proving Theorem \ref{thm:use}. 
\begin{proof}[Proof of Theorem \ref{thm:use}]
Reducing $(V,\nabla)$ modulo $q$, we get a vector bundle with connection on $\mathbb{G}_m$ over $K$, with almost all $p$-curvatures vanishing. After sufficient pullback, we may assume that $(V,\nabla)$ has a full set of solutions modulo $q$ (the $p$-curvature conjecture for $\mathbb{G}_m$ is known -- see \cite{Katz}, or \cite{Andre}, or \cite{Bost}). Therefore, changing basis by an element of $GL_n(K)[t,t^{-1}]$, suppose that the connection equals $\nabla(D) \textbf{e} = A \textbf{e}$, where $A \in M_n(\bOr{\X})$ satisfies $A \equiv \textbf{0} \mod{q}$. We now use an inductive argument to find a basis with respect to which the connection matrix $\nabla(D)$ is constant (we will use the exact same argument in the proof of the more general theorem too): assume that modulo $q^m$, the connection matrix is constant, and trivial modulo $q$. Then, modulo $q^{m+1}$, the connection matrix is of the form $[A_{0} + q^mB_m(t)]dt$, where $A_{0}$ has constant entries (which are $0$ modulo $q$), and $B_m(t)$ is a matrix with entries in $R[t^{-1},t]$, with no constant term. We pick a new set of coordinates
\begin{equation*}
{\bf f_{m+1}} = (I - q^mC_m(t)){\bf f_m}
\end{equation*}
where $\bf{f_m}$ is the old set of coordinates and where $C_m(t)$ satisfies 

\begin{equation*}
D(C_m(t)) = B_m(t)
\end{equation*}
(note that we assumed $B_m(t)$ has no constant term). 

Modulo $q^{m+1}$,
\begin{equation*}
\begin{array}{rcl}
 \nabla(D)(\bf{f_{m+1}}) &=& ((A_{0}t^{-1} + q^nB_m(t))(I - q^mC_m(t)) - q^m\displaystyle D(C_m(t)) \bf{f_m}\\[.15in]
&=&A_{0} (I - q^mC_m(t))\displaystyle\bf{f_m} \\[.15in]
&=& (I - q^mC_m(t))^{-1}A_{0} (I - q^mC_m(t))\displaystyle \bf{f_{m+1}}.

\end{array}
\end{equation*}
Therefore, with respect to the coordinates $\bf{f_{m+1}}$, the connection matix modulo $q^{m+1}$ is just $A_{0}$. The product \begin{equation*}
\prod\limits_{m=1}^{\infty}(I - q^mC_m(t))
\end{equation*}
clearly converges on the annulus $|t| = 1$. Therefore, $\nabla(D)\textbf{f} = A_0 \textbf{f}$, where $A_0 \in qM_n(K[|q|])$ and  $\textbf{f} = \prod\limits_{m=1}^{\infty}(I - q^mC_m(t)) \textbf{e}$. While we might need to invert infinitely many integers in order to change basis from $\textbf{e}$ to $\textbf{f}$, note that at each stage (that is, modulo $q^m$) we only invert finitely many primes. Therefore, with respect to $\textbf{f}$, for almost every prime ideal $\mathfrak{p}$, the vector bundle with connection has vanishing $p$-curvatures (modulo $q^m$). This is true for every $m$, but the set of primes to be omitted might grow with $m$. The following lemma implies that $A_0 = \textbf{0}$. The theorem follows.
\end{proof}

\begin{Lemma}\label{A_0}
Let $A_0 \in qM_n(K[|q|])$, and let $V$ be a vector bundle along with a basis $\textbf{f}$ and a connection $\nabla$, such that the connection matrix of $\nabla$ is $A_0$. Suppose that for each $m \in \N$, the $p$-curvatures of $\nabla$ vanish modulo ${q^m}$ for all primes outside a finite set (which could depend on $m$). Then $A_0 =\textbf{0}$ and $\textbf{f}$ is a basis of horizontal sections. 
\end{Lemma}
\begin{proof}
$A_0$ being $\textbf{0}$ modulo $q$, is topologically nilpotent, and therefore nilpotent modulo $q^m$. For large enough primes $p$, we have that the $p$-curvature (modulo $q^m$) vanishes, and that $A_0^p \equiv \textbf{0} \mod{q^m}$. We have 
\begin{equation*}
\begin{array}{rcl}
\displaystyle {\bf 0} &=& \nabla(D)^p - \nabla(D^p)\\[.12in]
\displaystyle & = & A_0^p - A_0 \\[.12in]
\displaystyle & = &A_0
\end{array}
\end{equation*}
with all equalities read modulo $(\mathfrak{p},q^m)$. This means that $A_0$ (mod $q^m$) is $\bf{0}$ modulo $\mathfrak{p}$ for almost all primes $\mathfrak{p} \subset \mathcal{O}$, and so $A_0 \equiv \textbf{0} \mod{q^m}$. Now letting $m$ tend to infinity, we see that $A_0 = \textbf{0}$. 

The claim about $\textbf{f}$ follows immediately. 
\end{proof}

\subsection{Proof of the $p$-curvature conjecture for rigid-analytic annuli}
We first prove a lemma. 

\begin{Lemma}\label{lem}
Let $(V,\nabla)$ be a vector bundle with connection on $\Xr$, such that $r_1 \le 1 \le r_2$. Suppose that the entries of the connection matrix (with respect to a chosen basis ${\bf e}$) $A = \sum A_n t^n$ satisfy $\eqref{growthcond}$. Suppose further, that $A$ modulo $q$ is constant and semisimple with integer eigenvalues. Then there exists
\begin{enumerate}
\item A basis $\textbf{f}$ of $V$ such that the connection matrix is trivial modulo $q$. 
\item A closed subannulus $\Yr$ of $\Xr$ with that $\Yr \supset \X$, such that the connection matrix with respect to $\textbf{f}$ satisfies \ref{growthcond} on $\Yr$. Further, $r'_1 < 1$ if $r_1 < 1$, and $r'_2>1$ if $r_2 >1$.
\end{enumerate}
\end{Lemma}
\begin{proof}
We remark that if $r_1 <1$ and $r_2 >1$, then $\eqref{growthcond}$ implies that $A$ modulo $q$ is constant. Suppose that $X$ is an element of $GL_n(K)$ which diagonalises $A$ modulo $q$. Pick ${\bf f_0} = X \bf{e}$. The connection matrix with respect to ${\bf f_0}$ is $X^{-1}AX$. Note that the entries (which we denote by $a_{kj}(t)$) of the new connection matrix (which we still denote by the letter $A$) satisfy $\eqref{growthcond}$ on the annulus with radii $r_i$, and that $A$ is now diagonal with integer entries, modulo $q$. 
Suppose that \\
\begin{equation*}
A_0 \equiv
\left(
\begin{array}{cccc}
d_0&&&\\
&d_1&&\\
&&\ddots & \\
&&&d_{n-1}
\end{array}
\right)
(mod \hspace{1 mm} q)
\end{equation*}\\
where the $d_i$ are the integer eigenvalues alluded to just above. Let $\textbf{f}$ be defined to be 
\begin{equation*}
{\bf f} =
\left(
\begin{array}{cccc}
t^{-d_0}&&&\\
&t^{-d_1}&&\\
&&\ddots & \\
&&&t^{-d_{n-1}}
\end{array}
\right)
{\bf f_0}
\end{equation*}\\
The connection matrix with respect to $\textbf{f}$ is trivial modulo $q$, thus proving the first assertion. We now demonstrate the existence of $\Yr$. Denote by $a'_{kj}(t)$ the entries of the new connection matrix. For $k \neq j$, we have 
\begin{equation*}
a'_{kj}(t) = t^{d_k - d_j}a_{kj}(t),
\end{equation*}
 and we have that 
\begin{equation*}
a'_{kk}(t) = a_{kk}(t) - d_k.
\end{equation*}

The $r_i = 1$ case follows directly from the previous two equations, therefore suppose that $r_1 < 1$. The $a'_{kk}(t)$ satisfy $\eqref{growthcond}$ on $\Xr$, and the $a'_{kj}(t)$ satisfy the strict version of $\eqref{growthcond}$ on $\X$. Fixing $k \neq j$, suppose that $a_{kj}(t) = \sum b_i t^i$. Then, $a'_{kj}(t) = t^{k-j} \sum b_i t^i$. Let  $m=k-j$. We have
\begin{equation*}
|b_i|r_1^i \le 1, \\
|b_i| < 1
\end{equation*}
Both quantities tend to zero as $|i|$ tends to infinity. To demonstrate the existence of $\Yr$, we need to find some $r'_1 < 1$ so that the following equations hold for all values of $i$: 
\begin{equation*}
|b_i|{r'_1}^{i+m} \le 1, \\
|b_i| < 1
\end{equation*}

Note that if $m \ge 0$, then the inequalities automatically hold with $r'_1 = r_1$. Assume therefore, that $m < 0$. For $i \ge -m$, any $r'_1$ would work. As $|b_i| < 1$ for all $i$, it is possible to choose an $r'_1<1$ such that  $|b_i|{r'_1}^{i+m}<1$, for $0 \le i \le -m$. This would make sure that the required inequalities hold for $i \ge 0$. Further refining our choice so that ${r'_1}^{-1+m} \le r_1^{-1}$, it follows that the required inequalities hold for negative $i$ too, thereby demonstrating the existence of $r'_1$. The exact same argument works for $r_2$ as well, proving the lemma.

\end{proof}

\begin{proof}[Proof of Theorem \ref{thm:main}]
Let the connection matrix be $A = \sum A_it^i$. Choose $ r_1 \le r \le r_2$, for $r =|q|^{\alpha}$ ($\alpha \in \Q$). By replacing $t$ by $t/q^{\alpha}$, we may assume that $r=1$. The constant term of the new connection matrix stays the same. By Corollary \ref{grow}, the entries of the connection matrix have coefficients in $K[|q|]$. Reducing modulo $q$, we obtain a vector bundle with connection on the ring $K[t^{-1},t]$, the coordinate ring of $\mathbb{G}_m$ over $K$. 

According to Katz, the vector bundle with connection on $\mathbb{G}_m$ over $K$ has regular singularities at 0 and $\infty$. Therefore, with respect to a cyclic basis, the entries of the connection matrix $\nabla(D)$ have to be regular functions on $\mathbb{P}^1$ (this follows directly from \cite[Theorem 11.9]{Katz}), and so $\nabla(D)$ has to be a constant matrix. Therefore, $A \equiv A_0 \mod{q}$. Note that this does not depend on the initial choice of $r$. 

The $p$-curvatures vanishing imply that $A_0$ modulo $q$ is semisimple with rational eigenvalues. Pulling back from the rigid annulus by the map $t \mapsto s^m$ for some $m$, a common multiple of all the denominators,  we get a vector bundle with connection on the rigid annulus with inner and outer radii $r_1^{1/m}$ and $r_2^{1/m}$ respectively. With respect to the same basis, and the derivation $D = s\frac{d}{ds}$, we see that the connection matrix gets multiplied by a factor of $m$. For ease of notation, we replace $s$ by $t$ and $r_i^{1/m}$ by $r_i$. $A_0$ is therefore semisimple with integer eigenvalues, modulo $q$. We again remark hat the pullback does not depend on the $r$ chosen.

We now solve the differential equation in a subannulus containing $\X$. There exists a subannulus $\Yr$ of $\Xr$, satisfying the conclusions of Lemma \ref{lem}. Note that the $p$-curvatures still make sense and vanish for almost every $p$. This is because both the change of basis matrices used in the lemma make sense and are invertible modulo almost every $p$. 

We now use the same inductive argument as in the proof of Theorem \ref{thm:use}, the base case of which has just been demonstrated: letting $\textbf{f}_m$ and $C_m$ be as before, we change basis by the matrix $I - q^mC_m(t)$ to make $\nabla(D)$ constant modulo $q^{m+1}$. Therefore, with respect to the basis $\textbf{f} = \prod\limits_{m=1}^{\infty}(I - q^mC_m(t)) \textbf{e}$, the connection matrix is constant and satisfies the hypothesis of Lemma $\ref{A_0}$. The connection has a full set of solutions (by Lemma \ref{A_0}) wherever the product $\prod\limits_{m=1}^{\infty}(I - q^mC_m(t))$ converges.
According to the lemma below, this converges on a subannulus containing $|t|=1$, with inner and outer radii $r''_i$, where $r''_1 < 1$ if $r'_1$ is, and $r''_2 > 1$ if $r'_2$ is. Hence, it is possible to cover $\Xr$ by a finite number of closed subannuli such that on each subannulus there exists a full set of solutions. These solutions clearly glue on the intersections and therefore form global solutions on the entire rigid annulus. 
\end{proof}

\begin{Lemma}
Notation as in the proof of the theorem. The infinite product $\prod\limits_{m=1}^{\infty}(I - q^mC_m(t))$ converges on a subannulus $\mathcal{X}_{r''_1,r''_2} \supset \X$. Further, $r''_1 < 1$ if $r'_1$ is, and $r''_2 > 1$ if $r'_2$ is. 
\end{Lemma}
\begin{proof}

We have that $\prod\limits_{m=1}^{\infty}(I - q^mC_m(t))$ converges on $\X$. We first prove that $q^mC_m(t)$ also satisfied the inequalities $\eqref{growthcond}$ with respect to the annulus with radii $r'_i$. Suppose that $_mA(t)$ is the connection matrix at the $m^{th}$ stage (i.e. with respect to the basis ${\bf f_m}$). If $q^mC_m(t)$ satisfies $\eqref{growthcond}$, then so do $D(q^mC_m(t))$ and $(I-q^mC_m(t))^{\pm 1}$. We have
\begin{equation*}
\displaystyle _{m+1}A(t) = (I-q^mC_m(t))^{-1} {_m}A(t)(I-q^mC_m(t)) + (I-q^mC_m(t))^{-1}D(q^mC_m(t))
\end{equation*}
Therefore, if $_mA(t)$ and $q^mC_m(t)$ satisfy $\eqref{growthcond}$, then $_{m+1}A(t)$ also does.

On the other hand, $_mA(t)$ doesn't have any non-constant terms where the power of $q$ is lower than $m$. Therefore, if $_mA(t)$ satisfies $\eqref{growthcond}$, so does $q^mB_m(t)$ , and so does $q^mC_m(t)$. As $_1A(t)$ clearly does satisfy $\eqref{growthcond}$, we have that all the $q^mC_m(t)$ do, too. 

To finish the proof, $q^nC_n(t)$ is bounded by $|q|^n$ on $|t|=1$, and by 1 on $|t| =r'_1$. The non-archimedian analogue of the Hadamard 3-circle theorem (\cite[Proposition 7.2.3]{Kedlaya}) implies that for every positive $\epsilon$, $q^mC_m(t)$ is bounded by $|q^{c_{\epsilon}}|^m$ on $|t|=r'_1+\epsilon$,  where $c_{\epsilon}$ is positive and independent of $m$. It follows that $\prod\limits_{m=1}^{\infty}(I - q^mC_m(t))$ converges on the subannulus with inner radius $r'_1 + \epsilon$ for any positive epsilon. The same argument works for $r'_2$ as well, and the lemma follows. 
\end{proof} 

\section{Families of holomorphic connections}
In this section, we give a rigid criterion for a holomorphic family of vector bundles with connections on certain families of annuli to have global solutions. 

Let $X$ and $E$ be a holomorphic annulus and disc with coordinates $t$ and $q$, with respect to which $X$ and $E$ are centered at $0$. All holomorphic functions on $X \times E$ can be expressed as series in $x, x^{-1}$ and $q$, which converge absolutely in any compact subset of $X \times E$. 

The field $\C((q))$ equipped with $\nu_q$, the $q$-adic valuation, is a complete non-archimedian field. As in \S 3, let $\X$ be the closed $q$-adic annulus over $\C((q))$ with inner and outer radii 1. 

Let $f \in \mathcal{O}^{\hol}(X\times E)$, $f(x,q) = \sum_{n \in \Z} a_n(q) x^n$ where the $a_n(q) \in \C[|q|]$. If $\nu_q(a_n(q)) \rightarrow \infty$ as $|n| \rightarrow \infty$, then $f$ converges on the rigid anulus $\X$, and we say that $f \in \bOr(\X)$ (notation as in section 3, with $R = \C[|q|]$). The main result we prove in this section is

\begin{Proposition}\label{csol} 
Let $V$ be a holomorphic rank $n$ vector bundle on $X \times E$, along with a connection $\nabla$ relative to $E$. Suppose that the connection matrix $A$ with respect to a basis $\textbf{e}$ has the following properties: 
\begin{enumerate}
\item $A$ has entries in $\bOr(\X)$. 
\item The vector bundle with connection on $\X$ whose connection matrix is $A$ has a full set of rigid solutions over $\bOr(\X)$. 
\end{enumerate}
Then $\nabla$ has a full set of solutions.
\end{Proposition}

\subsection{Preliminaries}
The data of $(V, \nabla)$ as in Proposition \ref{csol} is the same as a holomorphic family of connections $\nabla_{ X\times \{Q\}}$ on $V_{ X \times \{Q\}}$, parameterised by points $Q \in E$. A section $s$ is a solution if it is in the kernel of $\nabla$. The set of solutions  form an $\mathcal{O}^{\hol}(E)$-module. We first show the existence of local solutions. 

\begin{Lemma}\label{holex}
Let $Y$ be a simply connected open subset of $\C$, and $E_d$ be a holomorphic $d$-dimensional polydisc. Suppose that $(V, \nabla)$ is a vector bundle on $Y \times E_d$ equipped with a connection relative to $E_d$. There exists a full set of solutions for the connection $\nabla$. 
\end{Lemma}
\begin{proof}
It suffices to prove the lemma in the case where $Y$ is a disc. Suppose that the coordinate on $Y$ is $y$, with respect to which $Y$ is centered at 0. Let ${\bf q} = (q_1, \hdots, q_d)$ be a set of coordinates on $E_d$, with respect to which $E_d$ is centered at ${\bf 0} = (0,\hdots, 0)$.  Let $\textbf{e}$ be a basis of sections of $V$. We have that
$$\nabla(\frac{d}{dy})\textbf{e} = A(y,{\bf q})\textbf{e} = \sum_{i \ge 0} A_n({\bf q})y^n,$$
where the $A_n({\bf q})$ are power-series in the $q_i$ converging on $E_d$. Because $A(y,{\bf q})$ converges on $Y \times E_d$, the series $\sum_{i \ge 0} A_n({\bf q_0})y^n$ converges on $Y$ for every ${\bf q_0} \in E_d$. We claim that there exists a unique matrix $U(y,{\bf q})$ with entries in $\C[|y,q_1, \hdots ,q_d|]$ such that $U(y,{\bf q}) \equiv \textbf{I} \mod{y}$, and such that $U(y,{\bf q})\textbf{e}$ form a formal set of solutions for the connection. Indeed, this follows by writing $U(y,{\bf q}) = \sum_{i=0}^{\infty}U_i({\bf q})y^i$, and solving the equation $\nabla(\frac{d}{dy})U\textbf{e} = \textbf{0}$ recursively. The solution is seen to be 
$$U_{i+1}({\bf q}) = -\sum_{j=0}^{i}A_j(q)U_{i-j}({\bf q})/(i+1).$$
Note that if we specialized to a point ${\bf q} = {\bf q_0}$, $U(y,{\bf q_0})\textbf{e}$ would be a set of holomorphic solutions of $(V,\nabla)$ pulled back to $Y \times \{{\bf q_0}\}$ (this well known fact is also proved in \cite[Theorem 6.2.1]{Kedlaya}). This implies the convergence of $U(y,{\bf q})$ on all of $Y \times E_d$, proving the lemma. 
\end{proof}

For the rest of this section, we fix a point $P\in X$.  For every point $Q \in E$, the connection $\nabla_{X\times \{Q\}}$ gives the monodromy representation 
$$\rho_{(P,Q)}: \pi_1(X,P) \rightarrow \GL(V_{P \times Q})$$ 
Recall that we have fixed a basis $\textbf{e}$ of $V$, and so we can (and will) think of the monodromy representation as 
$$\rho_{(P,Q)}: \pi_1(X,P) \rightarrow \GL_n(\C).$$ 
Because $\pi_1(X,P) = \Z$, the data of $\rho_{(P,Q)}$ is the same as the data of a matrix $\mathcal{M}_{(P,Q)} \in \GL_n(\C)$, where $\mathcal{M}_{(P,Q)}$ is the image under $\rho_{(P,Q)}$ of the generator of $\pi_1$ given by the loop going around $0$ counterclockwise. By Lemma \ref{holex}, this gives a holomorphic map $\mathcal{M}_P$ from $E$ to $\GL_n(\C)$, or an element $\mathcal{M}_P \in \GL_n(\mathcal{O}^{\hol}(E))$. Notice that if we change basis by $U(x,q) \in \GL_n(\mathcal{O}^{\hol}(X \times E))$, $\mathcal{M}_P$ changes by conjugation by $U(P,q) \in \GL_n(\mathcal{O}^{\hol}(E))$. Henceforth, whenever we refer to monodromy matricies $\mathcal{M} \in \GL_n(\mathcal{O}^{\hol}(E))$, we will mean with respect to the basis $\textbf{e}$ and the point $P$ (unless explicitly stated otherwise). 
 
\begin{Proposition}
With the setup as above, the vector bundle with connection has a full set of holomorphic solutions if the monodromy matrix is the identity. 
\end{Proposition}
\begin{proof}
This follows trivially from the existence of local solutions, as posited by Lemma \ref{holex}. 
\end{proof}

%We have the following proposition for future use.
%\begin{Proposition}\label{futureuse} 
%Given a holomorphic family of connections on a family of annuli over a simply connected complex manifold $U$, there exists a holomorphic monodromy operator $\mathcal{M}: U \rightarrow \GL_n(\C)$ which fiberwise gives the monodromy of the connection. 
%\end{Proposition}
%\begin{proof}
%The conclusion of Lemma \ref{holex} holds with $E$ replaced by a higher-dimensional unit ball. 
%\end{proof}

\begin{Lemma}
Let $Y$ be a simply connected open subset of $\C$ and $V, \textbf{e}$ be a vector bundle on $Y \times E$ along with a basis $\textbf{e}$, equipped with a connection relative to $E$. Suppose that $\nabla_1$ and $\nabla_2$ are two connections whose connection matricies $A_i$ (with respect to the basis $\textbf{e}$) are congruent modulo $q^m$. Then, their solutions are also congruent modulo $q^m$. 
\end{Lemma}

\begin{proof}
The question is local in $Y$, so we assume that $Y$ is a disc, with coordinate $y$. Fix a point $P \in Y$. Let $g_1(y,q) \in \GL_n(\mathcal{O}(Y \times E))$ satisfy $-D(g_1) = A_1g_1$ (the existence of $g_1$ is guaranteed by Lemma \ref{holex}). Then, $-D(g_1) \equiv A_2g_1 \mod{q^m}$. Therefore, changing basis by $g_1$, we may assume that $A_1 = \textbf{0}$, and $A_2 \equiv \textbf{0} \mod{q^n}$. 

Suppose that $g_2(y,q) \in \GL_n(\mathcal{O}(Y \times E))$ satisfies $-D(g_2) = A_2g_2$. Replacing $g_2$ by $g_2(y,q)g_2^{-1}(P,q)$, we may assume that $g_2(P,q) = \textbf{I}$. As $A_2 \equiv \textbf{0} \mod q^n$, $g_2(y,q) = C(q) + q^ng_3(y,q)$. Evaluating at $y=P$, we see that $C(q) \equiv \textbf{I} \mod q^m$, whence $g_2 \equiv \textbf{I} \mod q^m$, proving the lemma. 
\end{proof}

\begin{Lemma}\label{cong}
Let $V,\textbf{e}$ be a vector bundle along with a basis on $X \times E$. Suppose that $\nabla_1$ and $\nabla_2$ are connections on $V$ relative to $E$. If the connection matricies of $\nabla_i$ are congruent modulo $q^m$, then $\mathcal{M}_1$ is congruent to $\mathcal{M}_2$ modulo $q^m$, where $\mathcal{M}_i$ is the family of monodromy representations corresponding to $\nabla_i$. 
\end{Lemma}
\begin{proof}
Let $\C_l = \C \setminus l$, where $l$ is some half-line originating at $0$. Define $X_0 \subset X$ and $X'_0 \subset X$ to be the open subsets obtained by intersecting $X$ with $\C_{l_0}$ and $\C_{l'_0}$ where $l_0$ and $l'_0$ are non-parallel half lines originating at $0$, and not passing through $P$. As $X_0$ and $X'_0$ are simply connected, the restriction of $\nabla_i$ to both these spaces have a full set of solutions. Let $P'$ be the point in the connected component of $X_0 \cap X'_0$ not containing $P$.   

Supppose that $\{s_{i}\}_{\alpha}$ ($\alpha$ ranging from 1 to $n$) form $\mathcal{O}(E)^{\hol}$-bases of solutions for the $\nabla_i$ pulled back to $X_0\times E$, such that the $\{s_1\}_{\alpha}$ are congruent to $\{s_2\}_{\alpha}$ modulo $q^m$. Similarly, let $\{s'_{i}\}$ be $\mathcal{O}^{\hol}$-bases of solutions for $\nabla_i$ pulled back to $X'_0 \times E$, which are congruent modulo $q^m$, and such that $\{s_{i}\} = \{s'_i\}$ on  the connected component of $X_0 \cap X'_0$ not containing $P$. The monodromy operator $\mathcal{M}_i$  corresponding to our chosen generator of $\pi_1(X)$ is precisely the element of $GL_n(\mathcal{O}(E))$ satisfying $\mathcal{M}(\{s'_i\}_{\{P\}\times E}) = \{s_i\}_{\{P\}\times E}$. By Lemma \ref{cong}, $\{s_1\}_{\{P\}\times E}$ is congruent to $\{s_2\}_{\{P\}\times E}$ and  $\{s'_1\}_{\{P\} \times E}$ is congruent to $\{s'_2\}_{\{P\} \times E}$, both modulo $q^m$. Therefore, it follows that $\mathcal{M}_1$ is congruent to $\mathcal{M}_2$ modulo $q^m$ as required. 
\end{proof}

\subsection{Constructing holomorphic solutions}
We now prove Proposition \ref{csol}, using the results proved in the previous subsection. 
\begin{proof}[Proof of Proposition \ref{csol}]
Let $\mathcal{M} \in GL_n(\mathcal{O}(E))$ be the monodromy matrix of $\nabla$. By hypothesis, there exists a matrix $g^{\rig} \in \GL_n(\bOr(\X))$ such that $D(g^{\rig}) + Ag = \textbf{0}$. For any $m \in \N$, there exists a matrix $g^{\hol}_m \in \GL(\mathcal{O}(E\times X))$, such that $g^{\hol}_m \equiv g^{\rig} \mod{q^m}$. Therefore, after changing basis of the vector bundle by $g^{\hol}_m$, the new connection matrix $A_m$ is congruent to $\textbf{0}$ modulo $q^m$. By Lemma \ref{cong}, $\mathcal{M}$ (with respect to the old basis) is conjugate to $\textbf{I}$  modulo $q^m$. Therefore, $\mathcal{M}$ equals $\textbf{I}$ modulo $q^m$. It follows that $\mathcal{M} = \textbf{I}$, and the proposition follows. 
\end{proof}

\subsection{Monodromy operators for more general families of connections}
Let $U$ be a path-connected complex manifold, and let $\A \rightarrow U$ be a family of annuli. Suppose that $(V,\nabla)$ is a rank-$n$ vector bundle on $\A$ with connection relative to $U$. Because we have not assumed that there is a section $U\rightarrow \A$, there needn't exist a global monodromy operator as in the case of \S 4.1. However, the following result does hold true: 
\begin{Lemma}\label{forfutureuse}
For every  $u \in U$, there exists a holomorphic map $\M: N_u \rightarrow \GL_n(\C)$  such that:
\begin{enumerate}
\item $N_u$ is a simply connected open neighbourhood of $u$. 
\item For every $u' \in N_u$, $\M(u')$ is the monodromy matrix of $(V, \nabla)$ restricted to $\A_{u'}$. 
\end{enumerate}
\end{Lemma}
Beause $N_u$ is simply connected, there exists a compatible choice of generators for $H_1(\A_{u'},\Z)$. We choose one such. For a vector bundle with connection on an annulus, the monodromy representation is only well defined up to conjugation. When we say that $\M_{u'}$ is the monodromy matrix of $(V, \nabla)$ restricted to $\A_{u'}$, we mean that the conjugacy class of $\M_{u'}$ equals the conjugacy class of the monodromy representation evaluated on the choice of generator of $H_1(\A_{u'},\Z)$. 

\begin{Proposition}\label{futureuse}
Setting as above. Suppose that there exists an open subset $U' \subset U$ such that for $u\in U'$, $(V,\nabla)$ pulled back to $\A_{u'}$ has finite monodromy, with the order $m$, which is independent of $u'$. Then for every $u \in U$, $(V,\nabla)$ pulled back to $\A_{u'}$ has finite monodromy. 
\end{Proposition}
\begin{proof}
This is a direct application of Lemma \ref{forfutureuse}. Let $u\in U$ be any point, and let $\beta$ be a path connecting $u$ and $u'$ ($u' \in U'$ some point). There exists a finite open cover $\{N_i \}$ of $\beta$ where each $N_i$ satisfies the conclusion of Lemma \ref{forfutureuse}. Assume that if we omit any $N_i$, we no longer have a cover. Let $\M_i$ be the corresponding monodromy operators. 

By holomorphicity, if any $N_i$ intersects $U'$, then $\M_i^m$ is the constant function ${\bf I}$. Further, if $N_i$ and $N_j$ intersect, and if $\M_i^m$ is ${\bf I}$, then so is $\M_j$. Some $N_i$ must intersect $U'$ (as $u' \in \beta$). Therefore, the connectedness of $\beta$ and the assumptions on $N_i$ give us that each $\M_i^m$ must equal ${\bf I}$, and so $\A_u$ has finite monodromy. As $u$ was arbitrary, the result follows. 
\end{proof}

We conclude this section by proving Lemma \ref{forfutureuse}. 

\begin{proof}[Proof of Lemma \ref{forfutureuse}]
Fix $u \in U$, and let $a \in \A_u$ be any point. By the implicit function theorem, there exists a neighbourhood of $a \in \A$ which is biholomorphic to $N_a \times X_a$, where $N_a \subset U$ is open and equals a polydisc, and $X_a$ is an open disc inside $\C$. By Lemma \ref{holex}, $(V,\nabla)$ has a full set of solutions on $N_a \times X_a$. Therefore, there exist local solutions to $\nabla$. 

Now, fix a point $a \in \A_u$, and a neighbourhood $N_a \times X_a$ as above, and let $\{e_i \}$ be a basis of solutions. As there exist local solutions, we may analytically continue the $\{e_i\}$ about the annulus, until we come back to the point $a$. Suppose that the continuation of the $\{e_i \}$, in some neighbourhood $N'_a \times X'_a$ of $a$, are $\{e'_i\}$. Let $X_u = X_a \cap X_a'$, and $N_u \subset N_a \cap N'_a$ be any simply connected neighbourhood of $u$. Because the sections $\{e_i \}$ and $\{ e'_i \}$ are holomorphic, there exists $\M \in \GL_n(\mathcal{O}^{\hol}(N_u \times X_u))$ such that $\M \{ e_i\} = \{ e'_i\}$. Because the sections are solutions of $\nabla$, $\M$ actually lives in $\GL_n(\mathcal{O}^{\hol}(N_u))$. This gives the required $\M: N_u \rightarrow \GL_n(\C)$.

\end{proof}
\section{Proof of the main technical result}

\subsection{Setup}
Notation is as in Theorem \ref{thm:tech}. Let $d$ be the dimension of $B$. Let $Y = \overline{B} \setminus B$ be the nodal locus. We may assume that $Y$ is a generically smooth divisor (and so, has dimension $d-1$). Indeed, blowing up at the point $s_0$ allows us to assume that the nodal locus is divisorial, and by normalizing, we may assume that it is generically smooth. We also assume that $\gamma_s$ contracts to the node $P \in \mathcal{C}_{s_0}$, where $s_0$ can be assumed to be a smooth point of $Y$. 

Let $F$ be the function field of $B$, and $K$ be the function field of $Y$. Suppose that the rational function $q \in F$ is a local uniformising parameter for $Y$. Let $S \subset F$ be the corresponding discrete valuation ring - the maximal ideal is generated by $q$, and residue field is $K$. The special fiber $\mathcal{C} \times \Spec K$ is a nodal curve with node $P$, and the generic fiber $\mathcal{C} \times \Spec F$ is a smooth curve. 

Let $\mathcal{C}^{/P}$ be the complete local ring  of $\mathcal{C}$ at $P$, and let $\mathcal{C}^{h,P}$ be the henselization of the local ring of $\mathcal{C}$ at $P$. There exists a positive integer $a$, and functions $x,y \in \mathcal{C}^{h,P}$ such that $\mathcal{C}^{/P} \simeq K[|x,y,q|]/(xy - q^a)$. 

\begin{definition}
A $\C$-point of $Y$ is said to be generic enough if it is induced by an embedding $K \hookrightarrow C$. 
\end{definition}

Pick $x_1, \hdots x_{d-1}$, elements in $K$ that form a transcendence basis. Then the $x_i$ and $q$ give holomorphic coordinates in a holomorphic neighbourhood in $B$ of any generic enough point $s \in Y(\C)$. Let $s_q$ be the $1$-dimensional manifold obtained by fixing the values of the $x_i$ and allowing $q$ to vary (such that $|q|$ is small enough). Pulling back $\mathcal{C}$ to $s_q$ gives a holomorphic one-parameter family of complex curves. Call this family $\mathcal{C}_{s_q}$. We first establish the result for this family. 

\begin{Lemma}
Notation as above. Let $\gamma \subset \mathcal{C}_{s_{q'}}$ (for some non-zero value of $q = q'$) be a simple closed loop which contracts to the node at $q = 0$. Then $\gamma$ has finite order (independent of the point $s$) in the monodromy representation. 
\end{Lemma}
\begin{proof}
We set things up so as to apply Theorem \ref{thm:use} and Proposition \ref{csol}. Let $x,y$ be as in the beginning of this section. Consider the $q$-adic formal scheme over $K[|q|]$ (the completion of $S$ at $q$) obtained by completing $\mathcal{C}$ along its special fiber. Suppose that $\mathscr{C}$ is the rigid space over $K((q))$ associated to the generic fiber of the formal scheme. 

The formal tube of $\mathscr{C}$ over $P$ is an open rigid anulus $\mathcal{X}^o$. The inner and outer radii of $\X^o$ are $|q|^a$ and $1$. We have that $\bOr(\mathcal{X}^o)= \mathcal{C}^{/P}$. The pullback of $(V,\nabla)$ to $\X^o$ descends to the ring $\mathcal{C}^{h,P}[1/q]$. By the cyclic basis theorem (for instance, see \cite[Theorem 4.4.2]{Kedlaya}), there exists a function $f \in \mathcal{C}^{h,P}[1/q]$, such that $(V,\nabla)$ pulled back to $\mathcal{C}^{h,P}[1/fq]$ is the trivial vector bundle with a basis $\textbf{e}$ cyclic for $\nabla, D= x\frac{d}{dx}$. 

Let $\alpha < 1$ be a positive rational number such that $f$ has no zeros in the rigid subannulus $\mathcal{X}'$ given by $|x| = q^{\alpha}$ of $\mathcal{X}^o$. By pulling back by a map of the form $q \mapsto q^m$, we assume that $\alpha$ is a positive integer. We also have that $\alpha < a$. (Recall that the $\X^o$ has inner radius $|q|^a$)

Let $t = x /q^{\alpha}$. Note that $t \in \mathcal{C}^{h,P}[1/q]$. Then $1/f$ has a series expansion in terms of $t, t^{-1}$ which converges on the rigid annulus $\X \subset \X^o$, given by $|t| = 1$. Notice that all the above data can be reduced modulo $p$ for almost all primes, because the functions $x,t,f$ are elements of $\mathcal{C}^{h,P}[1/q]$ and the basis $\textbf{e}$ is defined over $\mathcal{C}^{h,P}[1/q]$. Further, almost all $p$-curvatures vanish. 

A holomorphic neighbourhood of the node is described locally by the conditions $xy = q^a$, $|x|,|y|,|q| < \epsilon$. Again, $x,y$ are as above, and $\epsilon > 0$ is some real number that is small enough. In terms of $t$, the neighbourhood is now given by $0 < |q| < \epsilon,$ $|q|^{a-\alpha}/\epsilon < |t| < \epsilon/|q|^{\alpha}$. Suppose that $E$ is the open disc centered at $q=0$ with radius $\epsilon$, and let $E^{\times}$ equal $E \setminus \{0\}$. Choose $X'$ to be a holomorphic annulus centered at $t=0$, such that $X' \times E^{\times}$ is a subset of the above neighbourhood, and the series $f$ converges on $X' \times E^{\times}$. Fixing an appropriate value of $q$ with $|q| = \epsilon' \le \epsilon$, the series expansion of $1/f$ converges on a complex analytic subannulus $X$ of $X'$. This series has only finite denominator in $q$, because it converges on $\X$. Therefore, the series of $1/f$ has to converge for any non-zero value of $q$ with $|q| < \epsilon'$. Replace $\epsilon$ with $\epsilon'$. 

We have set things up so that the (cyclic) basis $\textbf{e}$ of $V$, and the entries of the connection matrix with respect to $\textbf{e}$ converge both on the rigid annulus $\X$ and the family of holomorphic annuli $X \times E^{\times}$, with at worst a pole of finite order at $q=0$. We now apply Theorem \ref{thm:use}, part 2, to deduce the existence of rigid solutions (after finite pullback). By \ref{thm:use}, part 1, the entries of the connection matrix doesn't have poles in $q$, and so $(V,\nabla)$ extends to $X \times E$. We may apply Proposition \ref{csol} to deduce that $(V ,\nabla)$ has finite monodromy. The result follows.
\end{proof}

\begin{Lemma}\label{lem:nbd}
Notation and setting as above. Then, in a complex analytic neighbourhood $U$ of $s_0$, the vanishing loop in each smooth fiber acts with finite order (independent of the fiber) in the monodromy representation. 
\end{Lemma}
\begin{proof}
The set of generic enough points of $Y$ are dense (in the analytic topology). Therefore, for an appropriate open neighbourhood $U$ of $s_0$, the set of points $s_q$ (where $s\in Y(\C)$ ranges over generic enough points in $Y$) are dense in $U$. By the above lemma, the vanishing loop acts with finite order (independent of the point) for a dense set of points in $U$. The proposition follows from the holomorphicity of $\nabla$, and Lemma \ref{forfutureuse}. 
\end{proof}

\subsection{Completion of the proof}
\begin{proof}[Proof of Theorem \ref{thm:tech}]
Let $\beta$ be a  path connecting $s$ to $s_0$ such that $\gamma_s$ can be pinched off along $\beta$. If $d$ (the dimension of $B$) is greater than $1$, we may assume that $\beta$ does not intersect itself. Otherwise, replace $\overline{B}$ with $\overline{B} \times \mathbb{P}^1$ and pull everything back by the projection to $B$ so that $d > 1$. 

Suppose that $U' \subset \mathcal{B}^{\an}$ is a contractible open set containing the intersection of $\beta$ with the interior of $B$. Then $\gamma_s$ defines a holomorphic family of annuli $X_{U'}$ over $U'$. Therefore, showing that $\gamma_s$ acts with finite order in the monodromy repsentation is the same as proving that $(V, \nabla)$ pulled back to $X_{U'}$ has finite monodromy. This is true by Lemma \ref{lem:nbd} and Proposition \ref{futureuse}. The theorem follows.
\end{proof}

\section{Applications}
\subsection{The universal family over $\Mod$}
 Let $\Mod$ be the moduli space of $d$ ordered points on $\p^1$, with $d>3$, and let $\mathcal{C} \rightarrow \Mod$ be the universal family. To account for the $\textrm{PSL}_2$ action on $\p^1$, we may assume that the first three marked points are $\{0,1,\infty\}$ in that order. $\Mod$ can be then identified with ${\p^1} ^{d-3} \setminus \Delta$, the locus $(Q_4,Q_5, \hdots Q_{d}) \in {\p^1} ^{d-3}$ where the $Q_i$ are all distinct and different from $\{0,1,\infty\}$. Let the ordered marked points on $\mathcal{C}$ be $Q_1, Q_2, \hdots, Q_d$ where $Q_1, Q_2, Q_3$ equal $0,1,\infty$ respectively. Note that $\mathcal{C}\rightarrow \Mod$ is defined over $\Z[1/m]$ for some appropriate $m$. Let $\Co$ denote the open subscheme $\mathcal{C}\setminus \{Q_1, \hdots Q_d\}$. Every simple closed loop $\gamma$ induces a partition of the set of $\{Q_1, \hdots, Q_d\} = S_1 \cup S_2$ into 2 subsets. Conversely, every partition into two sets is induced by some simple closed loop $\gamma$.  
\begin{definition}
Call a partition of $\{Q_1, \hdots Q_d\} = S_1 \cup S_2$ nice if $Q_1, Q_2, Q_3 \in S_1$ and $S_2 \neq \phi$. 
\end{definition}
We now prove Theorem \ref{P1}
\begin{proof}[Proof of Theorem \ref{P1}]
As in the proof of Theorem \ref{gen}, it suffices to prove the result in the case of {\it the} generic family. Suppose that $\Modbar$ is the compactification of $\Mod$ constructed in \cite{Knudsen}. The family $\mathcal{C} \rightarrow \Mod$ extends to a family $\Cbar \rightarrow \Modbar$. In words, the compactification can be described as follows: the non-compactness of $\Mod$ arises from the fact that the points $Q_i$ are not allowed to be the same. Therefore, to ``approach" the boundary of $\Mod$ is to ``let some of the points $Q_i$ collide" (also see \cite[Chapter 3, G]{Harris}. The boundary divisors correspond to 2-set partitions of the $\{Q_1,\hdots, Q_d\} = S_1 \cup S_2$ with $|S_i|>1$. If the partition is not nice assume without loss of generality that $|S_2 \cap \{0,1,\infty\}| = 1$. Regardless of the niceness of $S_2$, one can approach the boundary divisor corresponding to $S_1 \cup S_2$ by allowing the points of $S_2$ to collide (in case the partition isn't nice, allow the $Q_i$, $i>3$ to move towards the remaining $Q_2 \in S_2$). The family of curves degenarates with the loop inducing this partition getting pinched to a point. 

With this description in hand, we prove the theorem. Look at the partition induced by $\gamma$. If $|S_2| = \{1\}$, then Katz's original theorem applies. Otherwise, the discussion above shows that $\gamma$ can be pinched off to a point by approaching the boundary divisor corresponding to the partition. Therefore, Theorem \ref{thm:tech} applies to the family $\Co \rightarrow \Mod$, and the result follows. 

\end{proof}

\subsection{The generic $d$-pointed curve}
Suppose that $\Mgd$ denotes the moduli space of curves with $d$-marked points. If $g \geq 2$ or $d \geq 2$, there exists an open subscheme $U \subset \Mgd$ over which the moduli problem is fine. In this case let $\Mgdbar$ be the compactification of $\Mgd$.
In case $g = d = 1$, we consider the Legendre family of elliptic curves $y^2 = x(x-1)(x- \lambda)$ over $\mathbb{P}^1 \setminus \{0,1,\infty\}$. In either case, $\mathcal{C}$ be the family of curves over either $\Mgd$ or $\mathbb{P}^1 \setminus \{0,1,\infty\}$, with the marked points removed.  We now prove Theorem \ref{genpts} using Theorem \ref{thm:tech}

\begin{proof}[Proof of Theorem \ref{genpts}]
As in the case of Theorem \ref{gen}, it suffices to prove the result for {\it the} generic family, or the Legendre family (depending on whether $(g,d) = (1,1)$ or not). The proof mainly consists of combining the arguments used to prove Theorems \ref{P1} and \ref{gen}. We will therefore content ourselves with simply sketching the salient points of the proof. 

The mapping class group of a genus one curve with a single puncture acts transitively on the set of isotopy classes of non-trivial separating simple closed loops, and on the set of isotopy classes of non-separating simple closed loops. Therefore, by choosing a appropriate path to $0, 1$ or $\infty$, every simple closed loop in a smooth fiber of the Legnedre family can by contracted to a point.

Therefore, we assume that $d>1$, or $g>1$. Let $s \in U(\C)$ and $\gamma \subset \mathcal{C}^{\an}_s$ be a simple closed loop. Supppose first that the image of $\gamma$ in the fundamental group of the compactified curve is trivial. This means that $\gamma$ is separating, and one of the connected components of $\mathcal{C}_s \setminus \gamma$ is homeomorphic to the complex unit disc with finitely many punctures given by $Q_1 \hdots Q_i$. If $i = 0$ then $\gamma$ is contractible. If $i = 1$, then Katz's theorem applies. If $i> 1$, then approach the boundary of $\Mgd$ by allowing the points $Q_1 \hdots Q_i$ to ``come together". This boundary component of $\Mgdbar$ is fine. The family of curves degenerates with $\gamma$ getting pinched to a point, and the result follows by applying Theorem \ref{thm:tech}.

Otherwise, we proceed as in the proof of Theorem \ref{gen}. The results about the transitivity of the fundamental group acting on the isotopy classes of simple closed loops still hold - i.e. the action on non-separating simple closed loops is transitive, and two separating loops are in the same orbit if and only if homeomorphism class of the components for the first loop is the same as those for the second loop. In case $\gamma$ is non-separating, we degenerate to the boundary divisor $D_{\irr}$ consisting of irreducible nodal $d$-pointed curves, which is generically fine. Using the same method as in Theorem \ref{gen}, we may apply Theorem \ref{thm:tech} to prove the result for non-separating loops. 

 If $\gamma$ is separating, we use the boundary divisors $D_{g_1,d_1,g_2,d_2}$ given by reducible nodal curves, with components being two irreducible smooth $d_i$ pointed curves with genera $g_i$, satisfying $d_1 + d_2 = d$ and $g_1 + g_2 = g$ and $g_1 \leq g_2$. Unless $g_1 = 1$ and $d_1 = 0$, the generic point of this divisor is fine and the result follows by applying Theorem \ref{thm:tech}. 

Finally, if $g_1 = 1$ and $d_1 = 0$, we pull back to an appropriate cover of $\Mgdbar$ until the moduli problem becomes fine just as in Theorem \ref{gen}, and then apply Theorem \ref{thm:tech} appropriately. 
\end{proof}

\AtEndDocument{{\footnotesize%
  \textit{E-mail address}: \texttt{ashankar@math.harvard.edu} \par
  \textsc{Department of Mathematics, Harvard University, Cambridge, MA 02138,} \par \textsc{USA}
}}

\end{document}